\documentclass[12pt]{article}
\usepackage{textcomp}
\usepackage{amssymb}
\usepackage{latexsym}
\usepackage{amsthm}
\usepackage{amscd}
\usepackage{amsmath}
\usepackage{mathrsfs}
\usepackage{amsfonts}
\usepackage{enumitem}

\usepackage{lineno}
\usepackage{a4wide}
\usepackage{amsmath}
\usepackage{amssymb}
\usepackage{amsthm}
\usepackage{latexsym}
\usepackage{graphicx}
\usepackage[english]{babel}
\usepackage{makeidx}
\usepackage{microtype}

\newtheorem{thm}{Theorem}[section]
\newtheorem{lem}[thm]{Lemma}
\newtheorem{cor}[thm]{Corollary}
\newtheorem{defn-lem}[thm]{Definition-Lemma}

\newtheorem{prop}[thm]{Proposition}

\theoremstyle{remark}
\newtheorem{rem}[thm]{Remark}

\theoremstyle{definition}
\newtheorem{defn}[thm]{Definition}
\numberwithin{equation}{section}

\allowdisplaybreaks[4]

\def \Z{{\mathbb Z}}

\def\map#1.#2.{#1 \longrightarrow #2}
\def\rmap#1.#2.{#1 \dasharrow #2}

\DeclareMathOperator{\depth}{depth}
\DeclareMathOperator{\sdepth}{sdepth}
\DeclareMathOperator{\Hdepth}{Hdepth}

\def\fb#1.{\underset #1 \to \times}
\def\pr#1.{\Bbb P^{#1}}
\def\ring#1.{\mathcal O_{#1}}
\def\mlist#1.#2.{{#1}_1,{#1}_2,\dots,{#1}_{#2}}

\def\uloopr#1{\ar@'{@+{[0,0]+(-4,5)} @+{[0,0]+(0,10)}
@+{[0,0]+(4,5)}}^{#1}}

\def\dloopr#1{\ar@'{@+{[0,0]+(-4,-5)} @+{[0,0]+(0,-10)}
@+{[0,0]+(4,-5)}}_{#1}}

\def\rloopd#1{\ar@'{@+{[0,0]+(5,4)} @+{[0,0]+(10,0)}
@+{[0,0]+(5,-4)}}^{#1}}

\newcommand{\rdots}{\mathinner{\mkern1mu\raise1pt\hbox{.}\mkern2mu\raise4pt\hbox{.} \mkern2mu\raise7pt\vbox{\kern7pt\hbox{.}}\mkern1mu}} 

\def\lloopd#1{\ar@'{@+{[0,0]+(-5,4)} @+{[0,0]+(-10,0)}
@+{[0,0]+(-5,-4)}}_{#1}}

\long\def\ignore#1{}
\long\def\ignore#1{#1}

\begin{document}

\begin{center}
{\bf HILBERT SERIES AND HILBERT DEPTH OF SQUAREFREE VERONESE IDEALS}
\end{center}

\begin{center}
Maorong Ge, Jiayuan Lin and Yulan Wang
\end{center}

{\small {\bf Abstract} In this paper, we obtain explicit formulas for the Hilbert series and Hilbert depth of squarefree Veronese ideals in a standard graded polynomial ring.}

\section{Introduction}

In recent years, $\textit{depth, Stanley depth, Hilbert depth}$ and their relations of graded modules over polynomial ring $R=K[x_1,\cdots, x_n]$ have been extensively studied (e.g. \cite{Anw}-\cite{Uli}). One of the motivations behind these researches is the famous Stanley conjecture. It says that $\sdepth M \ge \depth M$ for all finitely generated graded $R$-module $M$. Although Stanley conjecture has been confirmed in some special cases, it is still widely open. Analogous to $\textit{depth and Stanley depth}$, W. Bruns et al \cite{Brun1} introduced $\textit{Hilbert depth}$.  For the reader's convenience, let us recall the definition of $\textit{Hilbert depth}$ from \cite{Brun1}.

\begin{defn}
Let $M$ be a finitely generated graded $S$-module. A $\textit{Hilbert decomposition}$ of $M$ is a finite family

$$\mathscr{H}=\big (S_i,s_i \big )$$

\noindent such that $s_i \in \Z^m$ (where $m=1$ in the standard graded case and $m=n$ in the multigraded case), $S_i$ is a graded $K$-algebra retract of $R$ for each $i$, and 

$$M \cong \underset{i}{\bigoplus} S_i(-s_i)$$

\noindent as a graded $K$-vector space.

The number $\Hdepth \mathscr{H}=\underset{i}{\min} \depth S_i(-s_i)$ is called the $\textit{Hilbert depth}$ of $\mathscr{H}$. The $\textit{Hilbert depth}$ of $M$ is defined to be

$$\Hdepth M=\max\{\Hdepth \mathscr{H}: \mathscr{H} \hskip .1 cm \textit{is a Hilbert decomposition of M}\}.$$

\end{defn} 

Note that a Stanley decomposition breaks $M$ into a direct sum of submodules over suitable subalgebras, while a Hilbert decomposition only requires an isomorphism to the direct sum of modules over such subalgebras. The latter one depends only on the Hilbert series of $M$. Since Stanley decompositions are Hilbert decompositions, Hilbert depth provides an upper bound for the Stanley depth. 

For a finitely generated standard graded $R$-module $M=\bigoplus_{k \in \Z} M_k$, the $\textit{Hilbert series}$ of $M$ is given by the Laurent series $H_M(T)= \sum_{k \in \Z} H(M,k) T^k$, where $H(M,k)=\dim_K M_k$ is the $\textit{Hilbert function}$ of $M$. We say a Laurent series $\sum_{k \in \Z} a_k T^k$ is $\textit{positive}$ if $a_k \ge 0$ for all $k$. It is easy to see that any Hilbert series is positive. In \cite{Uli}, J. Uliczka proved that the Hilbert depth of $M$ is equal to $\max \{r:(1-T)^r H_{M}(T) \hskip .1 cm \textit{positive}\}$. Using this, W. Bruns et al \cite{Brun2} computed the Hilbert depth of the powers of the irrelevant maximal ideal in $R$. In this paper we deal with Hilbert series and Hilbert depth of squarefree Veronese ideals $I_{n,d}$. We prove the following theorem.

\begin{thm}
Let $R=K[x_1, \cdots, x_n]$ be the standard graded polynomial ring in $n$ variables over a field $K$. Let $I_{n,d}$ be the squarefree Veronese ideal generated by all squarefree monomials of degree $d$ in $R$. Then

$(1)$ The Hilbert series of $I_{n,d}$ is $H_{I_{n,d}}(T)=\overset{n-1}{\underset{i=d-1}{\sum}} \binom{i}{d-1} T^d (1-T)^{-n+i-d+1}$;

$(2)$ The Hilbert depth of $I_{n,d}$ is equal to $d+\left \lfloor \binom{n}{d+1}/\binom{n}{d} \right\rfloor=d+\left \lfloor \frac{n-d}{d+1} \right\rfloor=d-1+\left \lceil \frac{n-(d-1)}{d+1} \right\rceil$.
\end{thm}

\begin{rem}
The (multigraded) Stanley depth of $I_{n,d}$ has been conjectured (see \cite{Cim2} and \cite{Keller}) to be equal to $d+\left \lfloor \binom{n}{d+1}/\binom{n}{d} \right\rfloor$. M. Keller et al \cite{Keller} and M. Ge et al \cite{Ge} partially confirmed this conjecture but it is still open. We wonder whether a suitable modification of a Hilbert decomposition $\mathscr{H}$ with $\Hdepth \mathscr{H}=d+\left \lfloor \binom{n}{d+1}/\binom{n}{d} \right\rfloor$ will bring us a solution to that conjecture.  
\end{rem}

To prove Theorem $1.2 (2)$, it is sufficient to show that $\max \{r:(1-T)^r H_{I_{n,d}}(T) \hskip .1 cm \textit{positive}\}$

\noindent $=d+\left \lfloor \binom{n}{d+1}/\binom{n}{d} \right\rfloor$. In order to do so, we need a nice formula for $H_{I_{n,d}}(T)$. Surprisingly, we cannot find such a formula in literature. In section $2$ we use generating functions to obtain an explicit formula for $H_{I_{n,d}}(T)$. We refer the reader to \cite{Wil} for the basic theory of generating functions.  In section $3$, we prove that $(1-T)^r H_{I_{n,d}}(T)$ is $\textit{positive}$ if and only if $r \le d+\left \lfloor \binom{n}{d+1}/\binom{n}{d} \right\rfloor$. This implies Theorem $1.2 (2)$.

\noindent Notations: In this paper,we use the generalized binomial coefficient $\binom{m}{r}=\frac{m(m-1)\cdots(m-r+1)}{r!}$ when $r \ge 0$. The binomial coefficient $\binom{m}{r}=0$ when $r<0$.

\section{Hilbert series of squarefree Veronese ideals}

Let $H_{I_{n,d}}(T)=\overset{\infty}{\underset{k=d}{\sum}} a_{n,d,k} T^k$ be the Hilbert series of the squarefree Veronese ideal $I_{n,d}$, where $a_{n,d,k}=\dim_K (I_{n,d})_k$. Then $a_{n,d,k}$ satisfying the following properties.

\begin{prop}

$(1)$ $a_{n,1,k}=\binom{n+k-1}{k}$;

$(2)$ $a_{n,d,k}=a_{n-1,d,k}+\overset{k-d+1}{\underset{s=1}{\sum}} a_{n-1,d-1,k-s}$ for $\min\{n,k\} \ge d \ge 2$; and

$(3)$ $a_{n,d,k}=a_{n-1,d,k}+a_{n-1,d-1,k-1}+a_{n,d,k-1}-a_{n-1,d,k-1}$ for $\min\{n,k\} \ge d \ge 2$.

\end{prop}

\begin{proof}

$(1)$ is easy because $I_{n,1}$ is just the irrelevant maximal ideal $\mathfrak{m}$ in $R$ and its Hilbert series is $H_{I_{n,d}}(T)=\overset{\infty}{\underset{k=d}{\sum}} \binom{n+k-1}{k} T^k$. To prove $(2)$ and $(3)$, let
 
$$P_{n,d,k}=\{u : u \in I_{n,d} \hskip .1 cm \text{monomial}, \deg u=k \hskip .1 cm \text{and} \hskip .1 cm x_{i_1} x_{i_2} \cdots x_{i_d}| u \hskip .1 cm \text{for some}\hskip .1 cm \{i_1, \cdots, i_d\} \subseteq [n] \}$$

Then $a_{n,d,k}=\dim_K (I_{n,d})_k=|P_{n,d,k}|$.

Let $x_1^s || u$ be the highest power of $x_1$ in $u \in P_{n,d,k}$. Then $s$ can take values $0, \cdots, k-d+1$. We can divide $P_{n,d,k}$ into $k-d+2$ classes $P_{n,d,k,s}$ according to these $s$ values. It is easy to see that $P_{n,d,k,s}$ are pairwise disjoint, $|P_{n,d,k,0}|=a_{n-1,d,k}$ and $|P_{n,d,k,s}|=a_{n-1,d-1,k-s}$ for $s>0$. So $(2)$ follows easily.

$(3)$ is an immediate consequence of $(2)$ because $a_{n,d,k-1}=a_{n-1,d,k-1}+\overset{k-d}{\underset{s=1}{\sum}} a_{n-1,d-1,k-s-1}$.
\end{proof}

Let $F(x,y,z)=\underset{\min\{n,k\} \ge d \ge 1}{\sum} a_{n,d,k} x^n y^d z^k$ be the generating function of $a_{n,d,k}$. Then the following is true.

\begin{prop}
$F(x,y,z)=\frac{xyz(1-z)}{(1-x-xyz-z+xz)(1-z-x)}$
\end{prop}

\begin{proof}

By the definition of $F(x,y,z)$ and Proposition $2.1 (3)$, we have that
$F(x,y,z)=\underset{\min\{n,k\} \ge d \ge 1}{\sum} a_{n,d,k} x^n y^d z^k=\underset{\min\{n,k\} \ge 1}{\sum} a_{n,1,k} x^n y z^k+\underset{\min\{n,k\} \ge d \ge 2}{\sum} a_{n,d,k} x^n y^d z^k$

$=y \underset{\min\{n,k\} \ge 1}{\sum} \binom{n+k-1}{k} x^n z^k+\underset{\min\{n,k\} \ge d \ge 2}{\sum} (a_{n-1,d,k}+a_{n-1,d-1,k-1}+a_{n,d,k-1}-a_{n-1,d,k-1}) x^n y^d z^k$

$=y \overset{\infty}{\underset{n=1}{\sum}} \overset{\infty}{\underset{k=1}{\sum}} \binom{n+k-1}{k}z^k x^n+ \underset{\min\{n,k\} \ge d \ge 2}{\sum} a_{n-1,d,k} x^n y^d z^k+\underset{\min\{n,k\} \ge d \ge 2}{\sum} a_{n-1,d-1,k-1} x^n y^d z^k+$

$\underset{\min\{n,k\} \ge d \ge 2}{\sum} a_{n,d,k-1} x^n y^d z^k-\underset{\min\{n,k\} \ge d \ge 2}{\sum} a_{n-1,d,k-1} x^n y^d z^k$

$=\frac{xyz}{(1-z-x)(1-x)}+x[F(x,y,z)-\frac{xyz}{(1-z-x)(1-x)}]+xyzF(x,y,z)+z[F(x,y,z)-\frac{xyz}{(1-z-x)(1-x)}]-xz[F(x,y,z)-\frac{xyz}{(1-z-x)(1-x)}]$.

Solving for $F(x,y,z)$ from the previous equation gives

$$F(x,y,z)=\frac{xyz(1-z)}{(1-x-xyz-z+xz)(1-z-x)}$$

\end{proof}

Now we deduce a nice formula for $H_{I_{n,d}}(T)$ from Proposition $2.2$.

\begin{thm}
The Hilbert series of the squarefree Veronese ideal $I_{n,d}$ is given by

$$H_{I_{n,d}}(T)=\overset{n-1}{\underset{i=d-1}{\sum}} \binom{i}{d-1} T^d (1-T)^{-n+i-d+1}$$
\end{thm}

\begin{proof}
By Proposition $2.2$, $F(x,y,z)=\underset{\min\{n,k\} \ge d \ge 1}{\sum} a_{n,d,k} x^n y^d z^k=\frac{xyz(1-z)}{(1-x-xyz-z+xz)(1-z-x)}$.

Now $\frac{xyz(1-z)}{(1-x-xyz-z+xz)(1-z-x)}=\frac{xyz}{1-z} \cdot \frac{1}{1-x-\frac{xyz}{1-z}} \cdot \frac{1}{1-\frac{x}{1-z}}=\frac{xyz}{1-z} \overset{\infty}{\underset{i=0}{\sum}} x^i (1+\frac{yz}{1-z})^i \overset{\infty}{\underset{j=0}{\sum}} \frac{x^j}{(1-z)^j}$

The coefficient of $x^n$ in $F(x,y,z)$ is $\frac{yz}{1-z} \underset{i+j=n-1}{\sum} (1+\frac{yz}{1-z})^i \frac{1}{(1-z)^j}=yz \overset{n-1}{\underset{i=0}{\sum}} (1+\frac{yz}{1-z})^i$

\noindent $\frac{1}{(1-z)^{n-i}}$. The coefficient of $y^d$ in $yz \overset{n-1}{\underset{i=0}{\sum}} (1+\frac{yz}{1-z})^i \frac{1}{(1-z)^{n-i}}$ is $\overset{n-1}{\underset{i=0}{\sum}} \binom{i}{d-1} z (\frac{z}{1-z})^{d-1}$

\noindent $\frac{1}{(1-z)^{n-i}}=\overset{n-1}{\underset{i=d-1}{\sum}} \binom{i}{d-1} z^d (1-z)^{-n+i-d+1}$.
Replacing $z$ with $T$ gives the formula $H_{I_{n,d}}(T)=$

\noindent $\overset{n-1}{\underset{i=d-1}{\sum}} \binom{i}{d-1} T^d (1-T)^{-n+i-d+1}$.
\end{proof}

As a corollary of Theorem $2.3$, we have

\begin{cor}

$a_{n,d,k}=\overset{n-1}{\underset{i=d-1}{\sum}} \binom{i}{d-1} \binom{n-i+k-2}{k-d}$.

\end{cor}

\begin{proof}
It follows directly from the expansion of $H_{I_{n,d}}(T)=\overset{n-1}{\underset{i=d-1}{\sum}} \binom{i}{d-1} T^d (1-T)^{-n+i-d+1}$.

\end{proof}

In the next section, we will prove that the Hilbert depth of $I_{n,d}$ is equal to $d+\left \lfloor \binom{n}{d+1}/\binom{n}{d} \right\rfloor$.

\section{Hilbert depth of squarefree Veronese ideals}

By J. Uliczka \cite{Uli}, the Hilbert depth of $I_{n,d}$ is equal to $\max \{r:(1-T)^r H_{I_{n,d}}(T) \hskip .1 cm \textit{positive}\}$. The coefficient of $T^{d+1}$ in $(1-T)^r H_{I_{n,d}}(T)$ is $a_{n,d,d+1}-r a_{n,d,d}$. By Corollary $2.4$, $a_{n,d,d}=\overset{n-1}{\underset{i=d-1}{\sum}} \binom{i}{d-1}= \binom{n}{d}$ and $a_{n,d,d+1}=\overset{n-1}{\underset{i=d-1}{\sum}} \binom{i}{d-1} (n-i+d-1)=n\overset{n-1}{\underset{i=d-1}{\sum}} \binom{i}{d-1}-\overset{n-1}{\underset{i=d-1}{\sum}} \binom{i}{d-1} (i-d+1)=n \binom{n}{d}-d\overset{n-1}{\underset{i=d-1}{\sum}} \binom{i}{d}=n \binom{n}{d}-d \binom{n}{d+1}=d \binom{n}{d}+(n-d) \binom{n}{d}-d \binom{n}{d+1}=d \binom{n}{d}+(d+1) \binom{n}{d+1}-d \binom{n}{d+1}=d \binom{n}{d}+\binom{n}{d+1}$. So 
 $a_{n,d,d+1}-r a_{n,d,d} \ge 0$ implies that $r \le \frac{a_{n,d,d+1}}{a_{n,d,d}}=\frac{d \binom{n}{d}+\binom{n}{d+1}}{\binom{n}{d}}=d+\frac{\binom{n}{d+1}}{\binom{n}{d}}$. Therefore, $r \le 
d+\left \lfloor \binom{n}{d+1}/\binom{n}{d} \right\rfloor$ is a necessary condition for the positivity of $(1-T)^r H_{I_{n,d}}(T)$. We will show that it is also a sufficient condition.

Denote $(1-T)^r H_{I_{n,d}}(T)=\overset{\infty}{\underset{k=d}{\sum}} b_{n,d,k,r} T^k$. The following lemma is true.

\begin{lem}
$b_{n,d,k,r}=\overset{n-1}{\underset{i=d-1}{\sum}} \binom{i}{d-1} \binom{n-i+k-r-2}{k-d}$.

\end{lem}

\begin{proof}

It follows directly from the expansion of $(1-T)^r H_{I_{n,d}}(T)=(1-T)^r \overset{n-1}{\underset{i=d-1}{\sum}} \binom{i}{d-1} T^d (1-T)^{-n+i-d+1}=\overset{n-1}{\underset{i=d-1}{\sum}} \binom{i}{d-1} T^d (1-T)^{-n+i+r-d+1}$.

\end{proof}

\begin{lem}
For all positive integers $\min \{n,k\} \ge d$, we have 

$(1)$ $b_{n-r+p-1,p,k,p-1}=\binom{n+k-r-1}{k}$ for any integer $1 \le p \le d$;

$(2)$ $b_{n,d,k,r}=(-1)^{k-d} \overset{r-k+1}{\underset{t=1}{\sum}} \binom{n-(d-1)-t+d-1}{d-1} \binom{r-(d-1)-t}{k-d}+\binom{n+k-r-1}{k}$.

\end{lem}

\begin{proof}

We use induction on $p$ to prove $(1)$. 

When $p=1$, we have that $b_{n-r+p-1,p,k,p-1}=b_{n-r,1,k,0}=a_{n-r,1,k}$, which is equal to  $\overset{n-r-1}{\underset{i=0}{\sum}} \binom{n-i+k-2}{k-1}=\binom{n+k-r-1}{k}$ by Corollary $2.4$. So $(1)$ holds true in this case.

Suppose that $b_{n-r+q-1,q,k,q-1}=\binom{n+k-r-1}{k}$ for any $1 \le q \le p$. If $p=d$, we are done. Otherwise $p \le d-1$. By Lemma $3.1$, we have that

$b_{n-r+p,p+1,k,p}=\overset{n-r+p-1}{\underset{i=p}{\sum}} \binom{i}{p} \binom{n-i+k-r-2}{k-p-1}=\overset{n-r+p-1}{\underset{i=p}{\sum}} \binom{i}{p} [\binom{n-i+k-r-1}{k-p}-\binom{n-i+k-r-2}{k-p}]$

$=\overset{n-r+p-1}{\underset{i=p}{\sum}} \binom{i}{p} \binom{n-i+k-r-1}{k-p}-\overset{n-r+p-1}{\underset{i=p}{\sum}} \binom{i}{p} \binom{n-i+k-r-2}{k-p}$

$=\overset{n-r+p-1}{\underset{i=p}{\sum}} \binom{i}{p} \binom{n-i+k-r-1}{k-p}-\overset{n-r+p-1}{\underset{i=p+1}{\sum}} \binom{i-1}{p} \binom{n-i+k-r-1}{k-p}$

$=\binom{n-p+k-r-1}{k-p}+\overset{n-r+p-1}{\underset{i=p+1}{\sum}} \binom{i-1}{p-1} \binom{n-i+k-r-1}{k-p}$

$=\binom{n-p+k-r-1}{k-p}+\overset{n-r+p-2}{\underset{i=p}{\sum}} \binom{i}{p-1} \binom{n-i+k-r-2}{k-p}$

$=\overset{n-r+p-2}{\underset{i=p-1}{\sum}} \binom{i}{p-1} \binom{n-i+k-r-2}{k-p}=b_{n-r+p-1,p,k,p-1}$, which is equal to $\binom{n+k-r-1}{k}$ by the inductive assumption.

To prove $(2)$, note that $b_{n,d,k,r}=\overset{n-1}{\underset{i=d-1}{\sum}} \binom{i}{d-1} \binom{n-i+k-r-2}{k-d}$

$=\overset{n-1}{\underset{i=n+k-r-1}{\sum}} \binom{i}{d-1} \binom{n-i+k-r-2}{k-d}+\overset{n+k-r-2}{\underset{i=d-1}{\sum}} \binom{i}{d-1} \binom{n-i+k-r-2}{k-d}$

$=(-1)^{k-d} \overset{n-1}{\underset{i=n+k-r-1}{\sum}} \binom{i}{d-1} \binom{r+i+1-n-d}{k-d}+\overset{n+k-r-2}{\underset{i=d-1}{\sum}} \binom{i}{d-1} \binom{n-i+k-r-2}{k-d}$

$\overset{i=n-t+d-1}{=} (-1)^{k-d} \overset{d+r-k}{\underset{t=d}{\sum}} \binom{n-t+d-1}{d-1} \binom{r-t}{k-d}+\overset{n+k-r-2}{\underset{i=d-1}{\sum}} \binom{i}{d-1} \binom{n-i+k-r-2}{k-d}$

$=(-1)^{k-d} \overset{r-k+1}{\underset{t=1}{\sum}} \binom{n-(d-1)-t+d-1}{d-1} \binom{r-(d-1)-t}{k-d}+\overset{n+k-r-2}{\underset{i=d-1}{\sum}} \binom{i}{d-1} \binom{n-i+k-r-2}{k-d}$

$=(-1)^{k-d} \overset{r-k+1}{\underset{t=1}{\sum}} \binom{n-(d-1)-t+d-1}{d-1} \binom{r-(d-1)-t}{k-d}+\overset{n+d-r-2}{\underset{i=d-1}{\sum}} \binom{i}{d-1} \binom{n-i+k-r-2}{k-d}$

$=(-1)^{k-d} \overset{r-k+1}{\underset{t=1}{\sum}} \binom{n-(d-1)-t+d-1}{d-1} \binom{r-(d-1)-t}{k-d}+\overset{n+d-r-2}{\underset{i=d-1}{\sum}} \binom{i}{d-1} \binom{(n+d-r-1)-i+k-(d-1)-2}{k-d}$

$=(-1)^{k-d} \overset{r-k+1}{\underset{t=1}{\sum}} \binom{n-(d-1)-t+d-1}{d-1} \binom{r-(d-1)-t}{k-d}+b_{n+d-r-1,d,k,d-1}$. By $(1)$, $b_{n+d-r-1,d,k,d-1}=\binom{n+k-r-1}{k}$. So $(2)$ follows.

\end{proof}

To prove the positivity of $(1-T)^r H_{I_{n,d}}(T)$ for $r \le 
d+\left \lfloor \binom{n}{d+1}/\binom{n}{d} \right\rfloor=d-1+\left \lceil \frac{n-(d-1)}{d+1}\right\rceil$, it is sufficient to show the following.

\begin{prop}
Let $n$ and $d$ be positive integers with $n \ge d$, and let $r=d-1+\left \lceil \frac{n-(d-1)}{d+1}\right\rceil$. Then for all $k=d+1, \cdots, r$ we have 

\begin{equation}
\binom{n+k-r-1}{k} \ge \overset{r-k+1}{\underset{t=1}{\sum}} \binom{n-(d-1)-t+d-1}{d-1} \binom{r-(d-1)-t}{k-d}
\end{equation}

\end{prop}

\begin{rem}
The inequality in Proposition is trivially true if $r+1 \le k$.
\end{rem}

\begin{proof} Replacing $r$, $n$ and $s$ in Proposition $3.1$ in \cite{Brun2} with $r-(d-1)$, $n-(d-1)$ and $d$ respectively, we have 

{\small
$$\binom{[n-(d-1)]+k-[r-(d-1)]-1}{k} \ge \overset{r-(d-1)}{\underset{t=1}{\sum}} \binom{n-(d-1)-t+d-1}{d-1} \binom{r-(d-1)-t}{k-d}$$}

Because $\binom{r-(d-1)-t}{k-d}=0$ for $r-k+1<t \le r-(d-1)$, after simplification, the above inequality becomes

$$\binom{n+k-r-1}{k} \ge \overset{r-k+1}{\underset{t=1}{\sum}} \binom{n-(d-1)-t+d-1}{d-1} \binom{r-(d-1)-t}{k-d}.$$

This completes the proof of Proposition $3.3$.
\end{proof}

By Lemma $3.2 (2)$, Proposition $3.3$, Remark $3.4$ and the necessary condition $r \le 
d+\left \lfloor \binom{n}{d+1}/\binom{n}{d} \right\rfloor$ for the positivity of $(1-T)^r H_{I_{n,d}}(T)$, we conclude that the Hilbert depth of $I_{n,d}$ is equal to $d+\left \lfloor \binom{n}{d+1}/\binom{n}{d} \right\rfloor=d+\left \lfloor \frac{n-d}{d+1} \right\rfloor=d-1+\left \lceil \frac{n-(d-1)}{d+1} \right\rceil$. This completes the proof of Theorem $1.2(2)$.

{}

{\footnotesize DEPARTMENT OF MATHEMATICS, ANHUI UNIVERSITY, HEFEI, ANHUI, 230039, CHINA}

$\textit{E-mail address:}$ ge1968@126.com

\medskip

{\footnotesize DEPARTMENT OF MATHEMATICS, SUNY CANTON, 34 CORNELL DRIVE, CANTON, 

\hskip .05 cm NY 13617, USA}

$\textit{E-mail address:}$ linj@canton.edu

\medskip

{\footnotesize DEPARTMENT OF MATHEMATICS, SUNY CANTON, 34 CORNELL DRIVE, CANTON, 

\hskip .05 cm NY 13617, USA}

{\footnotesize DEPARTMENT OF MATHEMATICS, ANHUI ECONOMIC MANAGEMENT INSTITUTE, HEFEI, 

\hskip .05 cm ANHUI, 230059, CHINA}

$\textit{E-mail address:}$ wangy@canton.edu
\end{document}